\documentclass{article}
\usepackage{amssymb}
\usepackage{amsmath}
\usepackage{amsfonts}
\usepackage{amscd}
\usepackage{epsfig}
\usepackage{amsmath,amstext,amsthm,amsbsy,amssymb}

\newcommand{\be}{{\bf e}}
\newcommand{\bc}{{\mathbb C}}
\newcommand{\br}{{\mathbb R}}

\newcommand{\bh}{{\mathbb H}}
\newcommand{\bi}{{\bf i}}

\newtheorem{thm}{Theorem}[section]
\newtheorem{lem}{Lemma}[section]
\newtheorem{pro}{Proposition}[section]
\newtheorem{cor}{Corollary}[section]

\newtheorem{defi}{Definition}[section]
\newtheorem{exam}{Example}[section]

\begin{document}

\title{The Moore-Penrose Inverses of Clifford Algebra $C\ell_2$}
\author{Rong lan ZHENG, Wen sheng CAO, Hui hui CAO }
\date{}
\maketitle

\bigskip
{\bf Abstract.} \,\, In this paper, we introduce a ring  isomorphism between the  Clifford algebra $C\ell_2$ and a ring of matrices, and represent the elements in $C\ell_2$ by real matrices. By  such a ring isomorphism, we  introduce the concept of the Moore-Penrose inverse in  Clifford algebra $C\ell_2$. we solve the linear equation $axb=d$, $ax=xb$ and $ax=\overline{x}b$.  We also obtain necessary and sufficient conditions for two numbers in $C\ell_2$ to be similar and pseudosimilar.

{\bf Mathematics Subject Classification.}\ \ Primary 15A24; Secondary  15A33.

{\bf Key words}\,\, Moore-Penrose inverse; linear equation; similar; pseudosimilar

\section{Introduction}

\qquad Based on Grassmann's exterior algebra, Clifford had earlier considered the multiplication rules of the Clifford algebra $C\ell_{0,n}$ in 1878. And Clifford created the multiplication rules of the Clifford algebra $C\ell_n$ in 1882. Finally Clifford generated Clifford algebra $C\ell_{p,q}$. Clifford algebra, also known as geometric algebra, has a wide range of applications in geometry and physics.

\begin{defi} The Clifford algebra  $C\ell_{p,q}$ with $p+q=n$  is generated by the orthonormal basis $\{i_1,\cdots,i_n\}$ of $\br^{p,q}$ with  the multiplication rules\cite{Lou}
\begin{equation}\label{mulrul}i_t^2=1,\,1\le t\le p,\,\,\,\,i_t^2=-1,\,p<t\le n,\,\,\,\,i_ti_m=-i_mi_t,\,t<m.\end{equation}
\end{defi}
Let $\br$, $\bc$, $\bh$ and $\bh_s$  be  respectively the real numbers, the complex numbers, the quaternions and the split quaternions.  Then we have $\bc\cong C\ell_{0,1}$, $\bh\cong C\ell_{0,2}$ and $\bh_s\cong C\ell_{1,1}$.

In this paper we focus on the Clifford algebra $C\ell_2$. The Clifford algebra $C\ell_2$ is a 4-dimensional real linear space with basis elements
\begin{equation}\label{basise}
	\be_0=1,\,\be_1=i_1,\,\be_2=i_2,\,\be_3=i_1i_2,
\end{equation}
which have the multiplication table
\begin{table}[h]\label{tab1}
	\centering
	\caption{Multiplication table for the Clifford algebra $C\ell_2$}
	\begin{tabular}{c|ccc}
		& $\be_1$ &  $\be_2$  &  $\be_3$\\
		\hline
		$\be_1$ &1 & $\be_3$ & $\be_2$ \\
		$\be_2$& -$\be_3$ &1& -$\be_1$ \\
		$\be_3$& -$\be_2$ & $\be_1$ & -1
	\end{tabular}
\end{table}

According to the multiplication rules, we have the following proposition.
 \begin{pro}
 \label{subalgebra}
	\item[(1)] $\bc=span\{1,\be_3\}$ and
	\begin{equation} C\ell_2=\bc+\bc \be_2.\end{equation}
	  and therefore each $a=a_0+a_1\be_1+a_2\be_2+a_3\be_3\in C\ell_2$ can be represented by
	\begin{equation}
			a=z_1+z_2\be_2,
	\end{equation}
	where $z_1=a_0+a_3\be_3$, $z_2=a_2+a_1\be_3\in \bc.$
	\item[(2)] Let $Cent(C\ell_2)=\{a\in C\ell_2:xa=ax,\forall x\in C\ell_2\}.$ we have   \begin{equation}Cent(C\ell_2)=\br.	\end{equation} 	
\end{pro}

\begin{defi}\label{defi1.2}
To each $a=a_0+a_1\be_1+a_2\be_2+a_3\be_3\in C\ell_2$ where $a_i\in \br,i=0,\cdots,3$, we define the following associated notations and maps of $a$:

the conjugate of $a$:\ \ $\bar{a}=a_0-a_1\be_1-a_2\be_2-a_3\be_3;$

the prime of $a$:\ \ $a'=a_0+a_1\be_1+a_2\be_2-a_3\be_3;$

the real part of $a$:\ \ $Cre(a)=(a+\overline{a})/2=a_0;$

the imaginary part of $a$:\ \ $Cim(a)=(a-\overline{a})/2=a_1\be_1+a_2\be_2+a_3\be_3;$

the modulus of $a$: $|a|=\sqrt{a_0^2+a_1^2+a_2^2+a_3^2}$;

the map $H$: $C\ell_2\longmapsto\br$: $H_a=\overline{a}a=a\overline{a}=a_0^2-a_1^2-a_2^2+a_3^2$;

the map $G$: $C\ell_2\longmapsto\br$: $G(a)=\Im(a)^2=a_1^2+a_2^2-a_3^2.$

\end{defi}

It is easy to verify the following proposition.
\begin{pro}\label{pro1.2}
\begin{itemize}
\item[(1)] Let $a,b\in C\ell_2$. Then $\overline{ab}=\bar{b}\bar{a},\,\,(ab)'=b'a',\,\,
H_{ab}=H_aH_b,\, H_{a'}=H_a$.
\item[(2)] Let $z \in \bc$. Then
   $\be_2 z=\overline{z}\be_2$.
    \item[(3)]  Let $a=z_1+z_2\be_2$, $z_1, z_2 \in \bc$. Then $H_a=|z_1|^2-|z_2|^2$, where $|z_1|^2=a_0^2+a_3^2,\,|z_2|^2=a_1^2+a_2^2$.

\end{itemize}
\end{pro}

The Moore-Penrose inverse plays an important role in the study of Clifford algebra. Cao\cite{cao}obtained the The Moore-Penrose inverse of $C\ell_{0,3}$ and studied the similarity and consimilarity in $C\ell_{0,3}$. Ablamowicz\cite{abla},  Cao and Chang\cite{caochang} used different algebra isomorphisms to find the Moore-Penrose inverse of split quaternions, which can be thought of as $C\ell_{1,1}$. Yildiz and Kosal\cite{onder}studied the semisimilarity and comsimilarity.

In this paper, we focus on the concepts of the Moore-Penrose inverse, similarity and pseudosimilar in Clifford algebra $C\ell_2$.

The paper is organized as follows.  In Section(\ref{matrixrep}),we will introduce a ring isomorphism between the Clifford algebra
$C\ell_2$ and a ring of matrices. By such a ring isomorphism, we will represent $a \in C\ell_2$ by a matrix L(a). In Section (\ref{mpinvsec}), by such a ring isomorphism, we introduce the concept of the Moore-Penrose inverse in Clifford algebra $C\ell_2$. In section (\ref{linear}), we solve the linear equation
$axb =d$. In Section (\ref{simsec}), we will obtain some necessary and sufficient conditions for two numbers in $C\ell_2$ to be similar and pseudosimilar.

\section{Ring isomorphism}\label{matrixrep}

\qquad Let $A^T$ be  the transpose of matrix $A$. Denote $\overrightarrow{x}=(x_0,x_1,x_2,x_3)^T\in \br^4$ for $x=x_0+x_1\be_1+x_2\be_2+x_3\be_3\in C\ell_2$. Each $a\in  C\ell_2$ define two maps form  $C\ell_2$ to $C\ell_2$ by
\begin{equation}\label{isomorph1}
	L_a:x\to ax
\end{equation}
and
\begin{equation}\label{isomorph2}
	R_a:x\to xa.
\end{equation}
Based on  multiplication rules for the basis of $C\ell_2$, the multiplication of two elements can be represented by an ordinary matrix-by-vector product. Gro{\ss}, Trenkler and Troschke\cite{gros} use this  method to study some properties of quaternions.  In such a way, we can verify the  following proposition.

\begin{pro} \label{promat}
For $a,x\in C\ell_2$, we  have
\begin{equation}\label{isomorph3}
	\overrightarrow{ax}=L(a)\overrightarrow{x}
\end{equation}
and
\begin{equation}\label{isomorph4}
	\overrightarrow{xa}=R(a)\overrightarrow{x},
\end{equation}
where
\begin{equation} \label{la}
	L(a)=\left(
	\begin{array}{cccccccc}
	a_0 &a_1  &a_2  &-a_3  \\
    a_1 &a_0   &a_3  & -a_2  \\
    a_2 &-a_3  &a_0   &a_1  \\
    a_3 &  -a_2 &a_1   & a_0
	\end{array}
	\right)
\end{equation}
and
\begin{equation} \label{ra}
	R(a)=CL(a)^TC=\left(\begin{array}{cccc}
	a_0 &a_1  &a_2  &-a_3  \\
	a_1 &a_0   & -a_3  & a_2  \\
	a_2 & a_3  &a_0   &-a_1  \\
	a_3 &  a_2 &-a_1   & a_0
\end{array}\right).
\end{equation}
where $C=diag(1,1,1,-1)$.
\end{pro}

\begin{pro}\label{pro2.2}Let $D=diag(1,-1,-1,1)$. Then we have
\begin{equation} \label{bara} L(\bar{a})=DL(a)^TD \end{equation}
and
\begin{equation} \label{Rbara} R(\bar{a})=DR(a)^TD. \end{equation}
\end{pro}

\begin{pro}\label{pro2.3} Let $E_n$ be the identity matrix of order $n$. For $a,b\in C\ell_2, \,\lambda\in \br$, we have
\begin{itemize}
  \item[(1)]  $a=b\Longleftrightarrow L(a)=L(b)\Longleftrightarrow R(a)=R(b),\,\, L(\be_0)=R(\be_0)=E_4;$
 \item[(2)] $L(a+b)=L(a)+L(b),\,\, L(\lambda a)=\lambda L(a),\,\, L(a')=L(a)^T;$
  \item[(3)]$ R(a+b)=R(a)+R(b), \,\,R(\lambda a)=\lambda R(a),\,\, R(a')=R(a)^T;$
\item[(4)]  $ L(a)R(b)=R(b)L(a),\,\,R(ab)=R(b)R(a),\,\,L(ab)=L(a)L(b).$
\end{itemize}
\end{pro}

Let $Mat(4,\br)$ be the set of real matrices of order $4$. By Proposition\ref{pro2.3}, $C\ell_2$ and $Mat(4,\br)$ can be thought of
as the rings $(C\ell_2,+,\cdot)$ and $( Mat(4,\br),+,\cdot)$. Then we have the following theorem.
\begin{thm}\label{ringiso}
Let $a \in C\ell_2$. Denot the map  $$L:C\ell_2\to Mat(4,\br)$$
 by
 \begin{equation}\label{isomorphmat}
	L: a\to L(a).
\end{equation}
Then $L$ is  a ring homomorphism from $(C\ell_2,+,\cdot)$ to $(Mat(4,\br),+,\cdot)$. Especially, let $im(L)$ be the image of such a
homomorphism. Then $$L:C\ell_{1,2}\to im(L)$$ is a ring isomorphism.
\end{thm}

\begin{pro}
	$\det(L(a))=\det(R(a))=H_a^2.$
\end{pro}
\begin{proof} By the definition\ref{defi1.2}, we have $$L(a\bar{a})=L(H_a)=H_aE_4.$$
Thus $$det(L(a\bar{a}))=H_a^4.$$
By the proposition\ref{pro2.2}, proposition\ref{pro2.3}, we have $$L(a\bar{a})=L(a)L(\bar{a})=L(a)DL(a)^TD.$$
 Thus $\det(L(a\bar{a}))=\det(L(a))^2$, thus $\det(L(a))=H_a^2$. Since $R(a)=CL(a)^TC$, we have $\det(L(a))=\det(R(a)).$
\end{proof}

\begin{defi}
The set of zero divisors in Clifford algebra $C\ell_2$ is
\begin{equation}
 Z(C\ell_2)=\{a\in C\ell_2:H_a=0\}.
 \end{equation}
\end{defi}

\begin{pro}
Let $a\in C\ell_2-Z(C\ell_2)$, then $H_a\neq 0$. The  inverse of $a$ is $$a^{-1}=\frac{\overline{a}}{H_a},$$ and $a^{-1}a=aa^{-1}=1$.
\end{pro}

\begin{exam}Let $a=1-\be_1+2\be_3$. Then $H_a=4$ and $$a^{-1}=\frac{1+\be_1-2\be_3}{4}.$$
\end{exam}

\section{The Moore-Penrose Inverse of Elements in $C\ell_2$}\label{mpinvsec}
By the theory of matrix, we have the following definition.
\begin{defi}
For any $m\times n$ real matrix $A$, if exist a $n\times m$ real matrix $X$, satisfy\cite{daih}
\begin{equation}\label{mpeqdefi}
AXA=A,XAX=X,(AX)^T=AX,(XA)^T=XA, \end{equation}
then $X$ be the matrix of Moore-Penrose inverse of $A$.
\end{defi}
We represent the Moore-Penrose inverse of $A$ by $A^+$.

Note that $L(a')=L(a)^T$, By Theorem\ref{ringiso} and the concept of Moore-Penrose inverse of real matrices, we have
the following Definition.
\begin{defi}\label{lemueq}
Let $a\in C\ell_2$. Then exist a $x\in C\ell_2$ satisfy:
\begin{equation}\label{mpeqs} axa=a,\,xax=x, \, (ax)'=ax, (xa)'=xa, \end{equation}
then $x$ be the Moore-Penrose inverse of $a$.
\end{defi}

By Theorem\ref{ringiso}and Lemma\ref{lemueq}, we have.
\begin{pro}
	\begin{equation}\label{laplus}L(a^+)=L(a)^+\end{equation}
	and
	\begin{equation}\label{raplus}R(a^+)=R(a)^+.\end{equation}
\end{pro}	

Obviously, $0^+=0$. If $H_a\neq0$, then $a^+=a^{-1}$. We have
$$L(a)^+=L\big(\frac{\overline{a}}{H_a}\big)=\frac{1}{H_a}\left(\begin{array}{cccc}
	a_0 &-a_1  &-a_2  &a_3  \\
	-a_1 &a_0   & -a_3  & a_2  \\
	-a_2 & a_3  &a_0   &-a_1  \\
	-a_3 &  a_2 &-a_1   & a_0
	\end{array}\right),\,  \forall a\in C\ell_2-Z(C\ell_2),$$
where $L(a)^+$ satisfies the equation(\ref{mpeqdefi}).

\begin{pro}\label{mppro}
For $a\in Z(C\ell_2)-\{0\}$. Let $$x=\frac{a'}{4(a_0^2+a_3^2)},$$ we have\begin{equation} axa=a,\,xax=x,\,(ax)'=ax,\, (xa)'=xa. \end{equation}
\end{pro}
\begin{proof} Since $ax=\frac{aa'}{4(a_0^2+a_3^2)},\,xa=\frac{a'a}{4(a_0^2+a_3^2)}$, by Proposition\ref{pro1.2}(1), we have\begin{equation}(ax)'=ax,
	\,\,\,(xa)'=xa.\end{equation} Let $a=z_1+z_2\be_2$, where $z_1=a_0+a_3\be_3$, $z_2=a_2+a_1\be_3$. Then $a'=\overline{z_1}+z_2\be_2$, $|z_1|^2=|z_2|^2$. By Proposition\ref{pro1.2}(2), we have $$e_2\overline{z_1}=z_1e_2,\, e_2z_2=\overline{z_2}e_2,\, e_2z_1=\overline{z_1}e_2.$$	
Then \begin{equation}\label{eq1} ax=\frac{1}{2}\big(1+\frac{z_2}{\overline{z_1}}\be_2\big),\end{equation}
\begin{equation}\label{eq2} xa=\frac{1}{2}\big(1+\frac{z_2}{z_1}\be_2\big).  \end{equation} Thus, we have $axa=a,\,xax=x.$
\end{proof}

By Proposition\ref{mppro}, we have $$L(a)^+=L\big(\frac{a'}{4(a_0^2+a_3^2)}\big)=\frac{1}{4(a_0^2+a_3^2)}\left(\begin{array}{cccc}
	a_0 &a_1  &a_2  &a_3  \\
	a_1 &a_0   & -a_3  & -a_2  \\
	a_2 & a_3  &a_0   &a_1  \\
	-a_3 &  -a_2 &a_1   & a_0
	\end{array}\right),\, \forall a\in Z(C\ell_2)-\{0\}, $$  where $L(a)^+$ satisfies the equation(\ref{mpeqdefi}).

By Proposition\ref{mppro}, we have the following Theorem.
\begin{thm}\label{defition} Let $a=a_0+a_1\be_1+a_2\be_2+a_3\be_3\in
	C\ell_2, \,a_i\in \br,i=0,\cdots,3$. Then
$$a^+=\left\{\begin{array}{ll}
0, & \hbox{ $a=0$;} \\
\frac{\overline{a}}{H_a}, & \hbox{ $H_a\neq 0$;} \\
\frac{a'}{4(a_0^2+a_3^2)}, & \hbox{ $H_a=0$ and $a\neq 0$.} \\
\end{array}
\right.$$
\end{thm}

We have another representation of the matrix of Clifford algebra $C\ell_2$.
\begin{thm}
Clifford algebra $C\ell_2$ as associative algebra, is  isomorphic to $\br^{2\times2}$\cite{Lou}. We have the map£º
\begin{equation}\label{isomorphm2}
	\varphi:a=a_0+a_1\be_1+a_2\be_2+a_3\be_3 \in C\ell_2 \to \varphi(a):=\left(\begin{array}{cc}
	a_0+a_1&a_2+a_3  \\
a_2-a_3&a_0-a_1   \\
		\end{array}\right).
	\end{equation}
\end{thm}

It's easy to verify that $det(\varphi(a))=H_a$. By the concept of the Moore-Penrose inverse, we have the following theorems.
\begin{thm}\label{thm3.3}
Let $a \in C\ell_2-Z(C\ell_2)$, then $H_a\neq0$. Let
		$$\varphi(a)^+=\varphi \big(\frac{\overline{a}}{H_a}\big)=\frac{1}{H_a}\left(\begin{array}{cccc}
	a_0-a_1 &-a_2-a_3   \\
	-a_2+a_3 &a_0+a_1
	\end{array}\right),$$
then $\varphi(a)^+$ satisfies the equation(\ref{mpeqdefi}).
\end{thm}

\begin{thm}\label{thm3.4}
Let $a \in Z(C\ell_2)-\{0\}$, then $H_a=0$. Let
		$$\varphi(a)^+=\varphi \big(\frac{a'}{4(a_0^2+a_3^2)}\big)=\frac{1}{4(a_0^2+a_3^2)}\left(\begin{array}{cccc}
	a_0+a_1 &a_2-a_3   \\
	a_2+a_3 &a_0-a_1
	\end{array}\right),$$
then $\varphi(a)^+$ satisfies the equation(\ref{mpeqdefi}).
\end{thm}

Thus we have the same definition of the Moore-Penrose inverse in Clifford algebra $C\ell_2$.

\section{Linear equation $axb=d$}\label{linear}

In this section, we will studied the linear equation $axb=d$ in $C\ell_2$. Then we choose the matrix of order 4 of a Clifford algebra. We need some identities. By Proposition\ref{pro2.3} and Definition\ref{lemueq}, we have the following proposition.
\begin{pro}\label{mpprop}  Let $a,b\in C\ell_2$. Then
 \begin{itemize}
	\item[(1)]	$L(a)L(a^+)L(a)=L(a), L(a^+)L(a)L(a^+)=L(a^+),L(a)L(a^+)=\big(L(a)L(a^+)\big)^T,L(a^+)L(a)=\big(L(a^+)L(a)\big)^T;$
	\item[(2)] $R(a)R(a^+)R(a)=R(a), R(a^+)R(a)R(a^+)=R(a^+),R(a)R(a^+)=\big(R(a)R(a^+)\big)^T,R(a^+)R(a)=\big(R(a^+)R(a)\big)^T;$
	\item[(3)] $\big(L(a)R(b)\big)^+=L(a^+)R(b^+).$
\end{itemize}
\end{pro}

By the By the theory of matrix, we have the following lemma.
\begin{lem}\label{general} Let $A\in \br^{m\times n},b\in
	\br^{m}$, linear equation $Ax=b$ is solvable if and only if $AA^{+}b=b$,  the general solution is\cite{daih}: $$x=A^{+}b+(E_n-A^{+}A)y,\, \forall y\in \br^{n}.$$
\end{lem}

\begin{thm}\label{thm4.1} Let $a=a_0+a_1\be_1+a_2\be_2+a_3\be_3,\, b=b_0+b_1\be_1+b_2\be_2+b_3\be_3\in  Z(C\ell_2)-\{0\}$ and $d\in  C\ell_2$. Then the equation $axb=d$ is solvable if and only if
 \begin{equation}\label{caxb1}
	\frac{aa'db'b}{16(a_0^2+a_3^2)(b_0^2+b_3^2)}=d,
	\end{equation}  the general solution is:  \begin{equation}\label{saxb1}
	x=\frac{a'db'}{16(a_0^2+a_3^2)(b_0^2+b_3^2)}+y-\frac{a'aybb'}{16(a_0^2+a_3^2)(b_0^2+b_3^2)},\,\forall y\in C\ell_2.
	\end{equation}
\end{thm}
\begin{proof}By Proposition\ref{pro2.3} and Theorem\ref{ringiso}, $axb=d$ is equivalent to $L(a)R(b)\overrightarrow{x}=\overrightarrow{d}$.  By Lemma\ref{general}, $axb=d$ is solvable if and only if $$L(a)R(b)\big(L(a)R(b)\big)^+\overrightarrow{d}=\overrightarrow{d}.$$    Returning to Clifford algebra form by Proposition\ref{mpprop}, we have $aa^+db^+b=d$. ¼´(\ref{caxb1}). By Lemma\ref{general},  the general solution of the equation $axb=d$ is $$\overrightarrow{x}=\big(L(a)R(b)\big)^+\overrightarrow{d}+\Big(E_4-\big(L(a)R(b)\big)^+L(a)R(b)\Big)\overrightarrow{y},\forall y\in C\ell_2.$$  Then the general solution can be expressed as $$x=a^+db^++(y-a^+aybb^+),\forall y\in C\ell_2.$$ That is(\ref{saxb1}).
\end{proof}

We have the following corollaries.

\begin{cor}\label{cor4.1}	
	Let $a=a_0+a_1\be_1+a_2\be_2+a_3\be_3\in  Z(C\ell_2)-\{0\}$ and $d\in  C\ell_2$.  Then the equation $ax=d$ is solvable if and only if $$\frac{aa'd}{4(a_0^2+a_3^2)}=d,$$ in which case all the solutions are given by $$x=\frac{a'd}{4(a_0^2+a_3^2)}+y-\frac{a'a}{4(a_0^2+a_3^2)}y,\,\forall y\in C\ell_2.$$
\end{cor}

\begin{cor}\label{cor4.2}	
	Let $a=a_0+a_1\be_1+a_2\be_2+a_3\be_3\in  Z(C\ell_2)-\{0\}$ and $d\in  C\ell_2$.  Then the equation $ax=0$ is solvable:  $$x=y-\frac{a'a}{4(a_0^2+a_3^2)}y,\,\forall y\in C\ell_2.$$
\end{cor}
		
\begin{cor}\label{cor4.3}
 Let $b=b_0+b_1\be_1+b_2\be_2+b_3\be_3\in  Z(C\ell_2)-\{0\}$ and $d\in  C\ell_2$.  Then the equation $xb=d$ is solvable if and only if $$\frac{db'b}{4(b_0^2+b_3^2)}=d,$$ in which case all the solutions are given by $$x=\frac{db'}{4(b_0^2+b_3^2)}+y-\frac{ybb'}{4(b_0^2+b_3^2)},\,\forall y\in C\ell_2.$$
\end{cor}

\begin{cor}\label{cor4.4}
Let $b=b_0+b_1\be_1+b_2\be_2+b_3\be_3\in Z(C\ell_2)-\{0\}$.  Then the equation$xb=0$ is solvable: $$x=y-\frac{ybb'}{4(b_0^2+b_3^2)},\,\forall y\in C\ell_2.$$
\end{cor}

We provide some examples as follows.
\begin{exam}
 Let $a=1+\be_2,\,b=\be_1+\be_3,\,d=1+\be_2$. Then $$a^+=\frac{1+\be_2}{4},\,\,b^+=\frac{\be_1-\be_3}{4},\,\,aa^+=a^+a=b^+b=\frac{1+\be_2}{2},\,\,bb^+=\frac{1-\be_2}{2},\,\,aa^+db^+b=d.$$ This case belongs to Theorem\ref{thm4.1}.  The solutions of $axb=d$  are given by
$$x=\frac{\be_1-\be_3}{4}+y-\frac{(1+\be_2)y(1-\be_2)}{4},\,\forall y\in C\ell_2.$$
\end{exam}

\begin{exam}
Let $a=1-\be_2,\,d=1+\be_1-\be_2+\be_3$. Then $$a^+=\frac{1-\be_2}{4},\,\,aa^+=a^+a=\frac{1-\be_2}{2},\,\,aa^+d=d. $$  This case belongs to Corollary\ref{cor4.1}.  The solutions of $ax=d$ are given by
$$x=\frac{1+\be_1-\be_2+\be_3}{2}+y-\frac{(1-\be_2)y}{2},\forall y\in C\ell_2.$$
\end{exam}

\begin{exam}
Let $a=1+\be_1+\be_2+\be_3$. Then $$ a^+a=\frac{1+\be_2}{2}.$$  This case belongs to Corollary\ ref{cor4.2}.  The solutions of $ax=0$ are given by
$$x=y-\frac{(1+\be_2)y}{2},\,\forall y\in C\ell_2.$$
\end{exam}

\begin{exam}
Let $b=\be_2+\be_3,\,d=1-\be_1$, Then $$b^+=\frac{\be_2-\be_3}{4},\,\,b^+b=\frac{1-\be_1}{2},\,\,bb^+=\frac{1+\be_1}{2},\,\,db^+b=d.$$
This case belongs to Corollary\ref{cor4.3}.   The solutions of $xb=d$ are given by
$$x=\frac{\be_2-\be_3}{2}+y-\frac{y(1+\be_1)}{2},\,\forall y\in C\ell_2.$$
\end{exam}

\begin{exam}
Let $b=2+\be_1+2\be_2+\be_3$. Then $$bb^+=\frac{1}{2}+\frac{2}{5}\be_1+\frac{3}{10}\be_2.$$
This case belongs to Corollary\ref{cor4.4}.  The solutions of $xb=0$ are given by
$$x=y-\frac{y(5+4\be_1+3\be_2)}{10},\,\forall y\in C\ell_2.$$
\end{exam}

\section{Similarity and Pseudosimilarity}\label{simsec}

It is well known that two quaternions are similar if and only if they have the same norm and real part. Such relationships were extended to other algebra systems.  In this section, we will studied the necessary and sufficient conditions for two elements in $C\ell_2$  to be similar and pseudosimilar.

\begin{defi}
	If $a,b\in C\ell_2$ are similar, then exist an element $u\in C\ell_2-Z(C\ell_2)$ such that $$au=ub.$$
\end{defi}

\begin{defi}
	If $a,b\in C\ell_2$ are pseudosimilar, then exist an element $u\in C\ell_2-Z(C\ell_2)$ such that $$au=\overline{u}b.$$

\end{defi}

Since $(L(a)-R(b))\overrightarrow{x}=\overrightarrow{0}$ is equivalent to
\begin{equation}\label{eq3}ax=xb.\end{equation}
Thus, we need studied $(L(a)-R(b))\overrightarrow{x}=\overrightarrow{0}$.

\begin{pro}\label{pro5.1}
The eigenvalues of $L(a)$ are $$\lambda_{1,2}=a_0+\sqrt{G(a)}$$ and $$\lambda_{3,4}=a_0-\sqrt{G(a)}.$$
\end{pro}
\begin{proof}Let $\lambda$ is the eigenvalues of $L(a)$, Then $\det(\lambda E_4-L(a))=0$. Since $L(\bar{a})=DL(a)^{T}D$, such a $\lambda$ is also a eigenvalue of $L(\bar{a})$. i.e. $\det(\lambda E_4-L(\bar{a}))=0$.
we have $$\det((\lambda E_4-L(a))(\lambda E_4-L(\bar{a})))=0.$$   By Proposition\ref{pro2.3}, we have $$\det(\lambda^2E_4-\lambda(L(a+\overline{a}))+L(a\overline{a}))=0.$$ That is $$\det((\lambda^2-2\lambda a_0+H_a)E_4)=0.$$
Hence\begin{equation}\label{eq4}\lambda^2-2\lambda a_0+H_a=0.\end{equation}
The solutions of the equation(\ref{eq4}) are  the eigenvalues of $L(a)$, and each eigenvalue occurs with algebraic multiplicity 2.
\end{proof}

Similarly, we have the following lemma.
\begin{lem}\label{lem5.1}
The eigenvalues of $R(b)$ are given by $$\omega_{1,2}=b_0+\sqrt{G(b)}$$ and $$\omega_{3,4}=b_0-\sqrt{G(b)}.$$
\end{lem}

 \begin{pro}\label{proped}
Let $F(a,b)=L(a)-R(b)$. Then the eigenvalues of$F(a,b)$ are
\begin{equation*}\lambda_{1,2,3,4}=a_0-b_0\pm \Big(\sqrt{G(a)}\pm \sqrt{G(b)}\Big),\end{equation*}
and
\begin{equation*}\det(F(a,b))=(a_0-b_0)^4-2(a_0-b_0)^2\Big(G(a)+G(b)\Big)+\Big(G(a)-G(b)\Big)^2.\end{equation*}
If $a_0=b_0$, $G(a)=G(b)$, then $\det(F(a,b))=0$, rank$(F(a,b))=2$.\\
If $a_0\neq b_0$, $\det(F(a,b))=0$, then $G(a), G(b)\geq0$ and rank$(F(a,b))=3$.
 \end{pro}

\vspace{2mm}

\begin{exam}
	Let $a=2+4\be_1+5\be_2,\,b=2+3\be_1+6\be_2+2\be_3$. Then $G(a)=G(b)=41$, $$F(a,b)=\left(\begin{array}{cccc}
	0 &1  &-1  &2  \\
	1 &0   &2  &-11  \\
	-1 & -2  &0   &7  \\
	-2 &  -11 &7   &0
	\end{array}\right),$$  and $\mathrm{rank}(F(a,b))=2$, the solutions of $\det(\lambda E_4-F(a,b))=0$ are $$\lambda_1=\lambda_2=0,\,\,\lambda_3=2\sqrt{41},\,\,\lambda_4=-2\sqrt{41}.$$
\end{exam}	

\begin{exam}
	Let $a=1+3\be_1+4\be_2-5\be_3,\,b=2+\be_1+\be_2+\be_3$. Then $G(a)=0,\,G(b)=1$, $$F(a,b)=\left(\begin{array}{cccc}
	-1 &2  &3  &6  \\
	2 &-1   &-4  &-5  \\
	3 & 4  &-1   &4  \\
	-6 &  -5 &4   &-1
	\end{array}\right),$$  and $\mathrm{rank}(F(a,b))=3$, the solutions of $\det(\lambda E_4-F(a,b))=0$ are $$\lambda_1=\lambda_2=0,\,\,\lambda_3=\lambda_4=-2.$$
\end{exam}

\begin{lem}\label{sginv}
Let $a,b\in C\ell_2-\br$, $F=L(a)-R(b)$. If $a_0=b_0,\,G(a)=G(b)$, the Moore-Penrose inverse of $F$ is  $$F^+=\frac{L(a')-R(b')}{2(|Cim(a)|^2+|Cim(b)|^2)}.$$
\end{lem}

\begin{proof}Let $a=a_0+a_1\be_1+a_2\be_2+a_3\be_3,\, b=a_0+b_1\be_1+b_2\be_2+b_3\be_3$ and $G(a)=G(b)$.
We have $$|Cim(a)|^2=a_1^2+a_2^2+a_3^2,\,\, |Cim(b)|^2=b_1^2+b_2^2+b_3^2.$$
 Then$$F=\left(\begin{array}{cccc}
0 &a_1-b_1  &a_2-b_2  &b_3-a_3  \\
a_1-b_1  &0 &a_3+b_3  &-a_2-b_2  \\
a_2-b_2  &-a_3-b_3&0  &a_1+b_1 \\
a_3-b_3  &-a_2-b_2 &a_1+b_1 &0
\end{array}\right).$$
Let
$$V=L(a')-R(b')=\left(\begin{array}{cccc}
0 &a_1-b_1  &a_2-b_2  &a_3-b_3  \\
a_1-b_1  &0&-a_3-b_3  &-a_2-b_2  \\
a_2-b_2  &a_3+b_3&0  &a_1+b_1 \\
b_3-a_3  &-a_2-b_2&a_1+b_1&0
\end{array}\right).$$
By direct calculation, we have $$FVF=2(|Cim(a)|^2+|Cim(b)|^2)F,\,\,VFV=2(|Cim(a)|^2+|Cim(b)|^2)V,\,\,(VF)^T=VF,\,\,(FV)^T=FV.$$
By the definition of Moore-Penrose inverse, we have $$F^+=\frac{L(a')-R(b')}{2(|Cim(a)|^2+|Cim(b)|^2)}.$$
\end{proof}

\begin{thm}\label{thmk0}
Let $a,b\in C\ell_2-\br$ and $a_0=b_0,\,G(a)=G(b)$.  Then the general solution of
linear equation $ax=xb$ is
\begin{equation}\label{thmk0e}x=y-\frac{a'ay-a'yb-ayb'+ybb'}{2(|Cim(a)|^2+|Cim(b)|^2)},\,\forall y\in C\ell_2.\end{equation}
\end{thm}
\begin{proof} Since $ax=xb$ is equivalent to $(L(a)-R(b))\overrightarrow{x}=\overrightarrow{0}$ , By Lemma\ref{general} and Lemma\ref{thmk0}, we have  $$\overrightarrow{x}=(E_4-F^+F)\overrightarrow{y}.$$ By Proposition\ref{pro2.3},  the general solution of equation is(\ref{thmk0e}).
\end{proof}

\begin{pro}\label{pro5.3}
 Let $a=a_0+a_1\be_1+a_2\be_2+a_3\be_3\in C\ell_2-\br$ and $G(a)<0$. Then there exists a $u \in C\ell_2-Z(C\ell_2)$ such that $$u^{-1}au=a_0+\sqrt{-G(a)}\be_3.$$
\end{pro}
\begin{proof}Consider the equation of Clifford algebra $C\ell_2$
\begin{equation}\label{eq5}ax=x(a_0+\sqrt{-G(a)}\be_3).\end{equation}
It easy to verify that $$x=a_2+(a_3-\sqrt{-G(a)}\be_1+a_1\be_3)$$ is a solution to equation(\ref{eq5}), if $a_3\leq0$. If $a_3>0$, $$x=a_3+\sqrt{-G(a)}+a_2\be_1-a_1\be_2$$ is a solution to equation(\ref{eq5}).
\end{proof}

\begin{exam}
Let $a=1-\be_3 \in C\ell_2$. Then$G(a)=-1$, the solution of equation$ax=x(1+\be_3)$ is $$x=-2\be_1,$$ where $H_x=-4.$
\end{exam}

\begin{exam}
Let $a=1+2\be_1+\be_2+3\be_3 \in C\ell_2$. Then $G(a)=-4$, the solution of equation $ax=x(1+2\be_3)$ is $$x=5+\be_1-2\be_2,$$ where $H_x=20.$
\end{exam}

\begin{pro}\label{pro5.4}
Let $a=a_0+a_1\be_1+a_2\be_2+a_3\be_3\in C\ell_2-\br$ and $G(a)>0$. Then there exists a $u \in C\ell_2-Z(C\ell_2)$ such that $$u^{-1}au=a_0+\sqrt{G(a)}\be_2.$$
\end{pro}
\begin{proof}Consider the equation of Clifford algebra $C\ell_2$
\begin{equation}\label{eq6}ax=x(a_0+\sqrt{G(a)}\be_2).\end{equation}
It easy to verify that $$x=a_3+(a_2-\sqrt{G(a)})\be_1-a_1\be_2$$ is a solution to equation(\ref{eq6}), if $a_2\leq0$. If $a_2>0$, $$x=a_2+\sqrt{G(a)}+a_3\be_1+a_1\be_3$$ is a solution to equation(\ref{eq6}).
\end{proof}

\begin{exam}
Let $a=1+5\be_1+3\be_3 \in C\ell_2$. Then $G(a)=16$, the solution of equation $ax=x(1+4\be_2)$ is $$x=3-4\be_1-5\be_2,$$ where $H_x=-32.$
\end{exam}

\begin{exam}
Let $a=1+2\be_1+\be_2-\be_3 \in C\ell_2$. Then $G(a)=4$, the solution of equation $ax=x(1+2\be_2)$ is $$x=3-\be_1+2\be_3,$$ where $H_x=12.$
\end{exam}

\begin{pro}\label{pro5.5}
Let $a=a_0+a_1\be_1+a_2\be_2+a_3\be_3\in C\ell_2-\br$ and $G(a)=0$. Then there exists a $u \in C\ell_2-Z(C\ell_2)$ such that $$u^{-1}au=a_0+\be_2+\be_3.$$
\end{pro}
\begin{proof} Consider the equation of Clifford algebra $C\ell_2$
\begin{equation}\label{eq7}ax=x(a_0+\be_2+\be_3).\end{equation}
Let $a_2=a_3=-1$, then $a_1=0$. $$x=\be_1$$ is a solution to equation(\ref{eq7}). If $a_2\neq a_3$ , $$x=a_1\be_1+(1+a_2)\be_2+(1+a_3)\be_3$$ is a solution to equation(\ref{eq7}).
\end{proof}

\begin{exam}
Let $a=1+3\be_1+4\be_2+5\be_3 \in C\ell_2$. Then $G(a)=0$, the solution of equation $ax=x(1+\be_2+\be_3)$ is $$x=3\be_1+5\be_2+6\be_3,$$ where $H_x=2.$
\end{exam}

By Proposition\ref{pro5.3}, Proposition\ref{pro5.4} and Proposition\ref{pro5.5}, we have the following lemma.

\begin{lem}\label{similar}
Let $a=a_0+a_1\be_1+a_2\be_2+a_3\be_3\in C\ell_2-\br$. Then there exists an element $u\in C\ell_2-Z(C\ell_2)$ such that $$u^{-1}au=\left\{                                                                                                                             \begin{array}{lll}                                                                                                                                    a_0+\sqrt{G(a)}\ \be_2, & \hbox{  $G(a)>0$;} \\                                                                                                                         a_0+\sqrt{-G(a)}\ \be_3, & \hbox{ $G(a)<0$;}\\                                                                                                                                  a_0+\be_2+\be_3,& \hbox{  $G(a)=0$.}
\end{array}                                                                                                                             \right.
$$
\end{lem}

Obviously, we have the following proposition.
\begin{pro}\label{pro5.6}
Let $a\in Cent(C\ell_2)$. Then $a,b \in C\ell_2$ are similar if and only if $a=b$.
\end{pro}

\begin{thm}\label{thmsim}
	Two elements $a,b \in C\ell_2-\br$ are similar if and only if
	$$Cre(a)=Cre(b),\,\, G(a)=G(b).$$
\end{thm}
\begin{proof} By Lemma\ref{similar}, the sufficiency can be proved.

Now to prove the necessity. If $a=a_0+a_1\be_1+a_2\be_2+a_3\be_3,\,b=b_0+b_1\be_1+b_2\be_2+b_3\be_3 \in C\ell_2$ are similar, then there exist an element $u \in C\ell_2-Z(C\ell_2)$ such that $u^{-1}au=b$, i.e. $u^{-1}(Cre(a)+Cim(a))u=Cre(b)+Cim(b)$. By Proposition
\ref{subalgebra}(2), we have $Cre(a)=Cre(b)$, thus $a_0=b_0$. Then
$$a\overline{a}u\overline{u}=au\overline{au}=ub\overline{ub}=b\overline{b}u\overline{u}.$$
Hence $$(a\overline{a}-b\overline{b})u\overline{u}=0.$$
Since $u\overline{u}=H_u\neq0$, we have $a\overline{a}=b\overline{b}$. Then $$a_0^2-a_1^2-a_2^2+a_3^2=b_0^2-b_1^2-b_2^2+b_3^2.$$  Thus $G(a)=G(b)$.
\end{proof}

Let $U=diag(1,-1,-1,-1)$, we have
\begin{equation}\label{eq8}ax=\overline{x}b\end{equation}
is equivalent to $\big(L(a)-R(b)U\big)\overrightarrow{x}=0.$
\begin{pro}\label{propsab}
 Let $a,b\in C\ell_2-\{0\},\,W(a,b)=L(a)-R(b)U$. Then  the eigenvalues of $W(a,b)$ are
$$\lambda_{1,2}=a_0 \pm \sqrt{\big(G(a)+H_b\big)},$$
$$\lambda_{3,4}=a_0+b_0\pm \sqrt{\big(G(a)+G(b)+2(-a_1b_1-a_2b_2+a_3b_3)\big)}.$$
and
$$\det(W(a,b))=(H_a-H_b)\big(H_a+H_b+2(a_0b_0+a_1b_1+a_2b_2-a_3b_3)\big)=(H_a-H_b)H_{\overline{a}+b}.$$	
 \end{pro}

\begin{lem}\label{lem1}
Let $a,b\in C\ell_2-\{0\},\,W(a,b)=L(a)-R(b)U$. When $\det(W(a,b))=0$, we have one of the following conditions.
\begin{itemize}
	\item [(1)] $H_a=H_b$, $\overline{a}+b=0$ and rank$(W(a,b))=1$.
\item [(2)]  $H_a=H_b$, $H_{\overline{a}+b}\neq 0$ and rank $(W(a,b))=3$.
\item [(3)]  $H_a=H_b$, $0\neq \overline{a}+b\in Z(C\ell_2)$ and rank$(W(a,b))=3$.
\item [(4)] $H_a\neq H_b$, $0\neq \overline{a}+b\in Z(C\ell_2)$ and rank $(W(a,b))=3$.
\end{itemize}

\end{lem}

We provide some examples.
\begin{exam}
Let $a=1+\be_1+\be_2+\be_3,\, b=-1+\be_1+\be_2+\be_3$. Then $\overline{a}+b=0,\,H_a=H_b=0,\,G(a)=G(b)=1$, $$W(a,b)=\left(\begin{array}{cccc}
	2&2  &2  &-2  \\
	0 &0   &0  & 0  \\
	0 &0   &0  & 0   \\
	0 &0   &0  & 0
	\end{array}\right).$$ Thus $\mathrm{rank}(W(a,b))=1$, the solutions of  $\det(\lambda E_4-W(a,b))=0$ are $$\lambda_1=2,\,\,\lambda_2=\lambda_3=\lambda_4=0.$$
\end{exam}

\begin{exam}
Let $a=2+3\be_1+4\be_2+5\be_3,\,b=5+3\be_1+4\be_2+2\be_3$. Then $H_{\overline{a}+b}=58,\,H_a=H_b=4,\,G(a)=0,\,G(b)=21$, $$W(a,b)=\left(\begin{array}{cccc}
	-3&6 &8  &-7  \\
	0 &7   &3  &0  \\
	0 &-3   &7  &0   \\
	3 &0   &0  & 7
	\end{array}\right).$$ Thus $\mathrm{rank}(W(a,b))=3$, the solutions of $\det(\lambda E_4-W(a,b))=0$ are $$\lambda_1=0,\,\,\lambda_2=4,\,\,\lambda_3=7+3\bi,\,\,\lambda_4=7-3\bi.$$
\end{exam}

\begin{exam}
Let $a=1+\be_1+\be_3,\,b=\be_3$. Then $H_{\overline{a}+b}=0,\,H_a=H_b=1,\,G(a)=0,\,G(b)=-1$, $$W(a,b)=\left(\begin{array}{cccc}
	1&1  &0  &-2  \\
	1 &1   &0  &0  \\
	0 &0   &1  & 1   \\
	0 &0   &1  & 1
	\end{array}\right).$$ Thus $\mathrm{rank}(W(a,b))=3$, the solutions of $\det(\lambda E_4-W(a,b))=0$ are $$\lambda_1=\lambda_3=2,\,\,\lambda_2=\lambda_4=0.$$
\end{exam}

\begin{exam}
 Let $a=1-\be_1+2\be_2-2\be_3,\,b=6+7\be_1+3\be_2+2\be_3$. Then $H_{\overline{a}+b}=0,\,H_a=0,\,H_b=-18,\,G(a)=1,\,G(b)=54$,  $$W(a,b)=\left(\begin{array}{cccc}
	-5&6  &5  &0  \\
	-8 &7   &-4  &1 \\
	-1 &4   &7  & -8  \\
	-4 &1   &-8  & 7
	\end{array}\right).$$ Thus $\mathrm{rank}(W(a,b))=3$, the solutions of $\det(\lambda E_4-W(a,b))=0$ are $$\lambda_1=1+\sqrt{17}\bi,\,\,\lambda_2=1-\sqrt{17}\bi,\,\,\lambda_3=14,\,\,\lambda_4=0.$$
\end{exam}

\begin{pro}\label{pro5.8}
Let $a,b\in C\ell_2-\{0\}$, $H_a=H_b$. Then there exist an element $u \in C\ell_2-Z(C\ell_2)$ such that $au=\overline{u}b$.
\end{pro}
\begin{proof}If $\overline{a}+b\neq0$, let$u=\overline{a}+b$. Since $H_a=H_b$, then $au=\overline{u}b$. By Lemma\ref{lem1}, one of the solutions to the equation $au=\overline{u}b$ is $$u=(\overline{a}+b)y,\,y\in \br.$$
$H_u\neq0$ if and only if $H_{\overline{a}+b}\neq0.$

If $\overline{a}+b=0$, we have $a_0=-b_0,\,a_1=b_1,\,a_2=b_2,\,a_3=b_3$, that is $b=-a_0+a_1\be_1+a_2\be_2+a_3\be_3$. Consider the equation
\begin{equation}\label{eq9} au=\overline{u}b.\end{equation}
It is easy to verify that $$u_1=a_3\be_1+a_1\be_3$$ is a solution of equation(\ref{eq9}), if $a_1^2-a_3^2\neq0$. $$u_2=a_3\be_2+a_2\be_3$$ is a solution of equation(\ref{eq9}), if $a_2^2-a_3^2\neq0$. $$u_3=a_3+a_0\be_3$$ is a solution of equation(\ref{eq9}), if $a_0^2+a_3^2\neq0$.
\end{proof}

Obviously, if $a=b=0$, there exist $x \in C\ell_2-Z(C\ell_2)$ such that $ax=\overline{x}b$. Thus $a,b$ are pseudosimilar.

\begin{thm}\label{thmcsim}
	Two elements $a,b \in C\ell_2-\{0\}$ are  pseudosimilar if and only if we have one of the following two conditions
	$$(1)\quad   \overline{a}+b=0;\  \quad  (2)\quad  H_a=H_b, H_{\overline{a}+b}\neq 0.$$
\end{thm}

\begin{proof} By Proposition\ref{pro5.8}, the sufficiency can be proved.
 Now to prove the necessity. If $a,b$ are pseudosimilar, then there exist an element $u\in  C\ell_2-Z(C\ell_2)$ such that $au=\overline{u}b$. Thus $$u\overline{u}a\overline{a}=au\overline{au}=\overline{u}b\overline{b}u=\overline{u}u\overline{b}b.$$ that is $H_u(H_a-H_b)=0$. Since $H_u\neq0$, we have $H_a=H_b$.
\end{proof}

\vspace{2mm}

{\bf Acknowledgements.}\quad This work is supported by Natural Science Foundation of China (11871379), Key project of  National Natural Science Foundation  of Guangdong Province Universities (2019KZDXM025)

\end{document}